\documentclass{amsart}
\usepackage{amsmath,amsthm,amssymb,bm}
\usepackage{color}
\usepackage{mathtools}
\mathtoolsset{showonlyrefs=true} %引用した式ã?み番号をå?力すã'? mathtools

\makeatletter
    
    \@addtoreset{equation}{section}
  \makeatother

\newtheorem{definition}{Definition}[section]

\newtheorem{theorem}[definition]{Theorem}
\newtheorem{lemma}[definition]{Lemma}
\newtheorem{corollary}[definition]{Corollary}
\newtheorem{remark}{Remark}[section]

\newcommand{\re}{\mathbb R} 

\def \pt{\partial}

\DeclareMathOperator{\supp}{supp}

\pagebreak

\title[Liouville-type theorem for the new Taylor--Couette flow]{
Liouville-type theorems \\
for the new Taylor--Couette flow \\
of the stationary Navier--Stokes equations
}
\author{Hideo Kozono, Yutaka Terasawa and Yuta Wakasugi}

\address[H. Kozono]{Department of Mathematics, Faculty of Science and Engineering,
Waseda University, Tokyo 169--8555, Japan, 
Mathematical Research Center for Co-creative Society, Tohoku University, 
Sendai 980-8578, Japan}
\email[H. Kozono]{kozono@waseda.jp, hideokozono@tohoku.ac.jp}
\address[Y. Terasawa]{Graduate School of Mathematics, Nagoya University,
Furocho Chikusaku Nagoya 464-8602, Japan}
\email[Y. Terasawa]{yutaka@math.nagoya-u.ac.jp}
\address[Y. Wakasugi]{Graduate School of Advanced Science and Engineering,
Hiroshima University,
Higashi-Hiroshima, 739-8527, Japan}
\email[Y. Wakasugi]{wakasugi@hiroshima-u.ac.jp}
\begin{document}
\begin{abstract}
We study the stationary Navier--Stokes equations
in the region between two rotating concentric cylinders.
We first prove that, under the small Reynolds number,  
if the fluid is axisymmetric and if its velocity is sufficiently small 
in the $L^\infty$-norm, 
then it is necessarily a {\it generalized} Taylor-Couette flow 
which is a {\it new exact solution} of the Navier--Stokes equations. 
If, in addition, the associated pressure is bounded or periodic in the 
$z$-axis, then it coincides with the well-known {\it canonical} 
Taylor-Couette flow.  
Next, we give a certain bound of the Reynolds number and 
the $L^\infty$-norm of the velocity such as the fluid is indeed, necessarily 
axisymmetric. 
It is clarified that smallness of Reynolds number of the fluid in the 
two rotating concentric cylinders governs both  
axisymmetry and the new exact form of the Taylor-Couette flow.   
\end{abstract}

\maketitle
\section{Introduction}
\footnote[0]{2010 Mathematics Subject Classification. 35Q30; 35B53; 76D05}

This paper concerns the
three-dimensional
stationary incompressible
Navier--Stokes equations
\begin{align}\label{ns0}
    \left\{ \begin{alignedat}{3}
    &v\cdot \nabla v + p = \nu \Delta v,\\
    &\nabla \cdot v = 0,
    \end{alignedat} \right.
\end{align}
where
$v = v(x) = (v_1(x),v_2(x),v_3(x))$
is the velocity vector,
$p = p(x)$
denotes the scalar pressure,
and $\nu > 0$ is the viscosity constant, respectively. 
For the Navier--Stokes equations \eqref{ns0}
in the whole space $\re^3$,
it has been an open problem
whether $v \equiv 0$ is the only solution
under the conditions that
$v$ has the finite Dirichlet integral and vanishes at the spatial infinity (see Galdi \cite[Remark X.9.4]{Ga}).
Seregin \cite{Se18} reformulated the problem in such a way whether any bounded solution $v$ must be constant.
There are many partial answers to this problem,  and for instance, 
we refer readers to
\cite{CaPaZh20, Ch14, ChWo16, ChJaLe21, KoNaSeSv09, KoTeWa17} and references therein.
\par
Recently, the Liouville-type theorems
in noncompact domains in $\re^3$
are also studied.
Carrillo-Pan-Zhang-Zhao \cite{CaPaZhZh20}
showed that a smooth solution with the finite Dirichlet integral to
the Navier--Stokes equations \eqref{ns0} in a
slab domain $\re^2 \times [0,1]$
with the no-slip boundary condition must be zero.
Among other results, they also treated the axially symmetric case 
with the periodic boundary condition and proved the Liouville-type result
under the finite Dirichlet integral.
The assumption of the finite Dirichlet integral was relaxed by Tsai \cite{Ts21}
and Bang-Gui-Wang-Xie \cite{BaGuWaXi}. 
In particular,  \cite{BaGuWaXi} obtained 
the Liouville-type theorem on the Poiseuille flow
of the Navier--Stokes equations \eqref{ns0}
in a slab domain $\re^2 \times [0,1]$
with no-slip boundary condition.
Indeed, they showed that if
$(v,p)$ is a smooth solution satisfying
$\| v \|_{L^{\infty}} < \pi$,
then $v$ must be the Poiseuille flow like 
$v = (ax_3(1-x_3), bx_3(1-x_3), 0)$
with some constants $a,b \in \re$. 
Their result may be regarded as the generalized Liouville-type theorem 
on the non-trivial flow. 
\par
Motivated by these results,  
we have reached a natural question
whether Liouville-type theorems hold
for other non-trivial exact solutions of the Navier--Stokes equations.
In this paper,
we study the Liouville-type theorem on the Taylor--Couette flow in
a region between two rotating concentric cylinders.
\par
%%%%%%%%%%%%
\subsection{Axially symmetric case}
Let
$0 < R_1 < R_2$
be constants
and let
$\Omega = \left\{ (x_1,x_2,x_3) \in \re^3 ; R_1 < 
\sqrt{x_1^2+x_2^2} < R_2 \right\}$,
that is, a region between two concentric cylinders.
In $\Omega$,
we consider the axially symmetric incompressible stationary Navier--Stokes equations in
the cylindrical coordinates:
\begin{align}
\label{ns}
    \left\{\begin{alignedat}{3}
    \left( v^r \partial_r + v^z \partial_z \right) v^r
    - \frac{(v^{\theta})^2}{r} + \partial_r p
    &=
    \nu \left( \partial_r^2 + \frac{1}{r} \partial_r
    + \partial_z^2 - \frac{1}{r^2} \right) v^r,\\
    \left( v^r \partial_r + v^z \partial_z \right) v^{\theta}
    + \frac{v^r v^{\theta}}{r}
    &=
    \nu \left( \partial_r^2 + \frac{1}{r} \partial_r
    + \partial_z^2 - \frac{1}{r^2} \right) v^{\theta},\\
    \left( v^r \partial_r + v^z \partial_z \right) v^z
    + \partial_z p
    &=
    \nu \left( \partial_r^2 + \frac{1}{r} \partial_r + \partial_z^2 \right) v^z,\\
    \frac{1}{r} \partial_r (r v^r) + \partial_z v^z
    &= 0,
    \end{alignedat}\right.
\end{align}
where
$r \in (R_1,R_2)$,
$z \in \re$,
$v = v(r,z) = v^r \bm{e}_r +  v^{\theta} \bm{e}_{\theta} + v^z \bm{e}_z$
with
$\bm{e}_r = (\cos \theta, \sin \theta, 0)$,
$\bm{e}_{\theta} = (-\sin \theta, \cos \theta, 0)$,  
$\bm{e}_z = (0,0,1)$
denoting the basis of the 
cylindrical coordinate,  
and
$p = p(r,z)$.
Moreover, we impose on $v$ the boundary conditions
\begin{align}\label{eq:bc}
	\begin{alignedat}{3}
    v^r(R_j,z) = v^z (R_j,z) = 0 \quad (j=1,2),\\
    v^{\theta}(R_1,z) = R_1 \omega_1, \quad
    v^{\theta}(R_2,z) = R_2 \omega_2
    \end{alignedat}
\end{align}
with some
$\omega_1, \omega_2 \in \re$,
that is, the inner and outer cylinders rotate with the angular velocities $\omega_1$ and $\omega_2$,
respectively.

It is well known that there exists
an exact solution to \eqref{ns} called
the Taylor--Couette flow:
\begin{align}\label{eq:TaCo}
    v = (0,v^{\theta},0)
    \quad \text{with} \quad
    v^{\theta}
    &=
    Ar + B\frac 1r
\end{align}
where
\begin{equation}\label{eq:mu:eta}
A = \left\{
\begin{array}{ll}
\displaystyle{ \frac{\mu - \eta^2}{1-\eta^2}\omega_1  }, &\quad \omega_1 \ne 0, \\
\displaystyle{ \frac{1}{1-\eta^2}\omega_2 }&\quad \omega_1 = 0,
\end{array}
\right.
\quad
B = \left\{
\begin{array}{ll}
\displaystyle{ \frac{1-\mu}{1-\eta^2}\omega_1 R_1^2  }, &\quad \omega_1 \ne 0, \\
\displaystyle{ -\frac{1}{1-\eta^2}\omega_2 R_1^2 }&\quad \omega_1 = 0
\end{array}
\right.
%\frac{\omega_{2}}{\omega_{1}},\quad \eta = \frac{R_{1}}{R_{2}}
\end{equation} 
with non-dimensional quantities $\mu$ and $\eta$ given by  
$$
\mu = \displaystyle{ \frac{\omega_2}{\omega_1} } \,\,\mbox{for $\omega_1\ne  0$}, \quad
\eta = \displaystyle{ \frac{R_1}{R_2} }.  
$$
It is also known that
the Taylor--Couette flow is stable if
$\omega_1$
is sufficiently small.
However, if $\omega_1$ exceeds a certain critical value,
then the Taylor--Couette flow becomes unstable
and a fluid motion so-called the Taylor vortex appears. 
See, e.g., Kirchg\"assner-Soger \cite{KiSo} and Chossat-Ioss \cite{ChIo}.  
For a recent result on the compressible fluid motion, we refer to Kagei-Teramoto \cite{KaTe}.     
\par
In this paper, we show a Liouville-type theorem 
on the more generalized Taylor--Couette flow including \eqref{eq:TaCo},
provided that the velocity is not so large. 
Our generalized Taylor--Couette flow below (\ref{eq:conc:thm}) 
is a {\it new exact solution} of the Navier-Stoeks equations.   
The first main theorem reads as follows:

%%%%%%%%%%%%%%%%%%%%%%%%%%%%%%%%%%%%%%%%%%%%%%%
\begin{theorem}\label{thm:liouville}
Let
$(v,p)$
be an axially symmetric smooth solution of \eqref{ns} in $\Omega$
with the boundary conditions \eqref{eq:bc}.
There exists a constant
$C_1(\nu,R_1,R_2)>0$
such that if
$\omega_1, \omega_2$
and
$\| v \|_{L^{\infty}}$
satisfy
\begin{align}\label{eq:ass:omega}
    \max\{ R_1 |\omega_1|, R_2 |\omega_2| \}
    < C_1(\nu,R_1,R_2)
    %\frac{\nu \pi}{2\sqrt{R_1R_2}(R_2-R_1)}
\end{align}
and
\begin{align}\label{eq:ass:v:Linf}
    \| v \|_{L^{\infty}(\Omega)}
    <
    C_1(\nu,R_1,R_2)
\end{align}
respectively,
then
$(v,p)$
must be the generalized Taylor--Couette flow
\begin{align}
\label{eq:conc:thm}
    \begin{alignedat}{3}
    v^r
    &\equiv
    0,\\
    v^{\theta}
    &=
    Ar + B\frac{1}{r},\\
    v^z
    &=
    \frac{a}{4\nu} R_{1}^{2} \left[
    \left(\frac{r}{R_{1}}\right)^{2} - 1 + \frac{1-\eta^{2}}{\eta^{2}\log \eta} \log \left( \frac{r}{R_1} \right)
    \right],\\
    p
    &=
    az + b
    + \frac{A^2}{2} r^2 
    + 2AB \log r
    - \frac{B^2}{2}\frac{1}{r^2}
    \end{alignedat}
\end{align}
with some constants $a, b \in \re$, where the constants $A$ and $B$ 
are the same as in \eqref{eq:mu:eta}.
In particular,
if the pressure $p$
is bounded or periodic in $z$,
then the constant $a$ in \eqref{eq:conc:thm} must be zero, and hence,
we have $v^z \equiv 0$, which means that 
$v$ coincides with the well-known canonical Taylor--Couette flow given by 
\eqref{eq:TaCo}.
\end{theorem}
%%%%%%%%%%%%%%%%%%%%%%%%%%%%%%%%%%%%%%%%%%%%
\begin{remark}
{\rm (i)}
Since the boundary condition \eqref{eq:bc} implies
$\max\{ R_1 |\omega_1|, R_2|\omega_2| \} \le \| v \|_{L^{\infty}}$,
the condition \eqref{eq:ass:omega} is necessary for \eqref{eq:ass:v:Linf}.

\noindent
{\rm (ii)}
From the proof of Theorem \ref{thm:liouville}, we may take
$C_1(\nu,R_1,R_2)$ as
\begin{align}\label{eq:C1}
    C_1(\nu,R_1,R_2) = 
    \frac{\nu}{2\sqrt{C_P}},
\end{align}
where
$C_P := \dfrac{R_2(R_2-R_1)^2}{R_1\pi^2}$. 
This implies that if the viscosity $\nu$ is large in comparison with the radii $R_1$ 
and $R_2$, then the fluid motion remains to be the laminar flow, 
i.e., the generalized Taylor-Couttee flow \eqref{eq:conc:thm}.   

\noindent
{\rm (iii)}
In the non-dimensional form of the equations \eqref{ns},
the Reynolds numbers $Re_j$ are defined by
\begin{align}
    Re_j = \frac{R_j\omega_j(R_2-R_1)}{\nu} \quad (j=1,2).
\end{align}
%(see Temam \cite[Chapter I\!I\, (4.5)]{Te}).
Then, \eqref{eq:C1} implies that the assumption \eqref{eq:ass:omega}
is written as
\begin{align}
    \max\{ Re_1, Re_2 \} &< \frac{R_2-R_1}{\nu}C_1(\nu,R_1,R_2) = \frac{\pi \sqrt{R_1}}{2\sqrt{R_2}}.
\end{align}
Namely, Theorem \ref{thm:liouville} implies that
if the Reynolds numbers and the velocity are bounded by 
a certain constant determined only by means of radii $R_1$ and $R_2$, 
then the axisymmetric flow must be necessarily the generalized Taylor--Couette flow \eqref{eq:conc:thm} which is a new exact solution of the Navier-Stokes equations.  
\par
\noindent
{\rm (iv)}
The assumption that the pressure $p$
is bounded or periodic in $z$ seems physically reasonable, and hence 
in a possible physical situation, 
the laminar axially symmetric flow $v$ in the two rotating concentric cylinder   
is necessarily the canonical Taylor--Couette flow \eqref{eq:TaCo}.
\par
\noindent
{\rm (v)}
Temam \cite[Ch. I\!I, Section 4]{Te} 
studied the uniqueness and the non-uniqueness of the problem \eqref{ns}
in the case of $\omega_2 = 0$ provided that the flow $v$ of (\ref{ns}) is periodic in $z$. 
Introducing the disturbance $u=(u^r, u^\theta, u^z)$ such that $v$ has the form like 
$v = v_0 + u$  with $v_0$ denoting the 
Taylor--Couette flow \eqref{eq:TaCo}, he reduced such a question on uniqueness 
whether $u\equiv 0$ in $\Omega$.  
It was proved in \cite[Ch. I\!I, Proposition 4.2]{Te} that,  under smallness hypotheses 
of the Reynolds number,
$u \equiv 0$ provided that $(u^r, u^z)$ is written by the stream function $\psi$ 
in the coordinate $(r, z)$ satisfying $\pt_r\psi(r_1, z)=\pt_r\psi(r_2, z)=0$. 
In comparison with Temam's  result, 
we remove assumption of periodicity in $z$ and avoid to make use of 
such a stream function,
although we impose on $v$ the smallness condition \eqref{eq:ass:v:Linf}.  
\end{remark}

%%%%%%%%%%%%%%%%%%%%%%%%%%%%%5
\subsection{General case}
Next, we treat the general case
in which the axial symmetry is not necessarily assumed.
Consider the same region
$\Omega$
as in Section 1.1,
and the incompressible stationary
Navier--Stokes equations in the
cylindrical coordinates in
$\Omega$:
\begin{align}
\label{ns2}
    \left\{\begin{alignedat}{3}
    \left( v\cdot \nabla \right) v^r
    - \frac{(v^{\theta})^2}{r} + \partial_r p
    &=
    \nu \left( \Delta - \frac{1}{r^2} \right) v^r
    - \nu \frac{2}{r^2} \partial_{\theta} v^{\theta},\\
    \left( v\cdot \nabla \right) v^{\theta}
    + \frac{v^r v^{\theta}}{r}
    + \frac{1}{r} \partial_{\theta} p
    &=
    \nu \left( \Delta - \frac{1}{r^2} \right) v^{\theta}
    + \nu \frac{2}{r^2}\partial_{\theta} v^r,\\
    \left( v\cdot \nabla \right) v^z
    + \partial_z p
    &=
    \nu \Delta v^z,\\
    \ \frac{1}{r} \partial_r (r v^r)
    + \frac{1}{r} \partial_{\theta} v^{\theta} 
    + \partial_z v^z
    &= 0,
    \end{alignedat}\right.
\end{align}
where we have used the notations
\begin{align}
\label{eq:nabla:rtz}
    \left( v\cdot \nabla \right)
    &= v^r \partial_r + \frac{v^{\theta}}{r} \partial_{\theta}
    + v^z \partial_z,\\
\label{eq:Delta:rtz}
    \Delta
    &= \partial_r^2 + \frac{1}{r} \partial_r + \frac{1}{r^2} \partial_{\theta}^2 + \partial_z^2. 
\end{align}
Moreover, we also impose on $v$ the same boundary conditions as \eqref{eq:bc}, that is,
\begin{align}\label{eq:bc:2}
	\begin{alignedat}{3}
    v^r(R_j,\theta,z) = v^z (R_j,\theta,z) = 0 \quad (j=1,2),\\
    v^{\theta}(R_1,\theta,z) = R_1 \omega_1, \quad
    v^{\theta}(R_2,\theta,z) = R_2 \omega_2
    \end{alignedat}
\end{align}
with some $\omega_{1}, \omega_{2} \in \re$.

For the general case,
under similar assumptions to
\eqref{eq:ass:omega} and \eqref{eq:ass:v:Linf},
we have the following Liouville-type 
theorem for
$(\partial_{\theta}v,\partial_{\theta}p)$
which shows axial symmetry of
the solutions to \eqref{ns2}.

%%%%%%%%%%%%%%%%%%%%%%%%%%%5
\begin{theorem}\label{thm:liouville:2}
Let
$(v,p)$
be a smooth solution of \eqref{ns2} in $\Omega$
with the boundary conditions \eqref{eq:bc:2}.
There exists a constant
$C_2(\nu,R_1,R_2)>0$
such that if
$\omega_1, \omega_2$
and
$\| v \|_{L^{\infty}}$
satisfy
\begin{align}
    \max\{ R_1 |\omega_1|, R_2 |\omega_2| \}
    <
    C_2(\nu,R_1,R_2)
\end{align}
and
\begin{align}
    \| v \|_{L^{\infty}(\Omega)}
    <
    C_2(\nu,R_1,R_2)
\end{align}
respectively,
then it holds that 
$\partial_{\theta}v \equiv 0$ and  $\partial_{\theta}p \equiv 0$ in $\Omega$, 
that is,
$(v,p)$ is axially symmetric.
\end{theorem}
%%%%%%%%%%%%%%%%%%%%%%%%%%%

\begin{remark}
From the proof of Theorem \ref{thm:liouville:2},
we may take the constant
$C_2(\nu, R_1, R_2)$
as
\begin{align}
    C_2(\nu, R_1, R_2)
    = \nu \left( \sqrt{C_P} \left( 2 + \frac{C_P}{R_1^2} \right) + \frac{3C_P}{2R_1} \right)^{-1},
\end{align}
where
$C_P := \dfrac{R_2(R_2-R_1)^2}{R_1\pi^2}$.
\end{remark}

Combining Theorems \ref{thm:liouville} and \ref{thm:liouville:2},
we immediately reach the following
Liouville-type theorem for the general
case.
%%%%%%%%%%%%%%%%%%%%%%%%%%
\begin{corollary}\label{cor}
Let $(v,p)$
be a smooth solution of \eqref{ns2} in $\Omega$
with the boundary conditions \eqref{eq:bc:2}.
Let
$C_1(\nu,R_1,R_2)$ and $C_2(\nu,R_1,R_2)$ be the same 
constants as in Theorems \ref{thm:liouville} and \ref{thm:liouville:2}, respectively. 
We set $C_\ast(\nu,R_1,R_2) \equiv \min\{ C_1(\nu,R_1,R_2), C_2(\nu,R_1,R_2) \}$. 
Suppose that $(v,p)$ is a smooth solution of \eqref{ns2} in $\Omega$
with the boundary conditions \eqref{eq:bc:2}. 
If $\omega_1, \omega_2$ and $v$ satisfy
\begin{align}
    \max\{ R_1 |\omega_1|, R_2 |\omega_2| \}
    <
     C_\ast(\nu,R_1,R_2)
\end{align}
and
\begin{align}
    \| v \|_{L^{\infty}(\Omega)}
    <
    C_\ast(\nu,R_1,R_2),
\end{align}
respectively,
then
$(v,p)$ is axially symmetric
and coincides with the generalized Taylor--Couette flow given by 
\eqref{eq:conc:thm}. 
In particular, if $p$ is bounded or periodic in $z$, then $v$ is necessarily 
the canonical Taylor-Couette flow (\ref{eq:TaCo}).
\end{corollary}
%%%%%%%%%%%%%%%%%%%%%%%%%%%%
\begin{remark} 
It is easy to see that the generalized Taylor-Couette flow (\ref{eq:conc:thm}) 
is also a solution of the Stokes equations in $\Omega$ with the same 
boundary condition (\ref{eq:bc}).  Hence without any assumption on smallness of 
$\|v\|_{L^\infty(\Omega)}$ as in (\ref{eq:ass:omega}) it holds that 
any bounded smooth solution $v$ of the Stokes equations uniquely coincides with 
the generalized Taylor-Couette flow (\ref{eq:conc:thm}) .  
In particular, if the pressure $p$ is bounded or periodic in $z$, 
then $v$ is necessarily the canonical Taylor-Couette flow (\ref{eq:TaCo}).  
This may be regarded as the Liouville type theorem on the Stokes equations.       
\end{remark}

%%%%%%%%%%%%%%%%%%%%%%%%%%%%%%%%%%%%%%%%%%%%%%%%
%%%%%%%%%%%%%%%%%%%%%%%%%%%%%%%%%%%%%%%%%%%%%%%
\section{Preliminaries}
In what follows,
$C$ denotes generic constants
which may change from line to line.
Also, the operator
$\nabla_{r,z}$
and $\nabla_{r,\theta,z}$
stand for
$\nabla_{r,z} f(r,z) = (\partial_r f, \partial_z f)(r,z)$
and
$\nabla_{r,\theta,z}f(r,\theta,z)
= (\partial_r f, \partial_{\theta} f, \partial_z f)(r,\theta,z)$,
respectively.

We first state the boundedness of derivatives of solutions.
This can be proved in
the same way as \cite[Lemma 2.3]{BaGuWaXi}
in which the three dimensional slab domain is treated,
since all estimates in the proof are local and
do not depend on the shape of $\Omega$.
%%%%%%%%%%%%%%%%%%%%%%%%%%%%%%%%%%%%%%%%%%%%
\begin{lemma}\label{lem:bdd}
Let $(v,p)$ be a smooth solution of
the Navier--Stokes equations \eqref{ns0} in $\Omega$
with the boundary conditions \eqref{eq:bc}.
Assume that
$v$ is bounded.
Then,
$\nabla_{r,\theta,z} v, \nabla_{r,\theta,z}^2 v$,
and
$\nabla_{r,\theta,z}p$
are also bounded.
\end{lemma}
%%%%%%%%%%%%%%%%%%%%%%%%%%%%%%%%%%%%%%%%%%%%%

Next, we prepare the test function used in this paper.
Let $L>1$ and define
\begin{align}\label{varphiL}
    \varphi_L(z) =
    \begin{cases}
    1 &(|z| < L-1),\\
    L-|z| &(L-1 \le |z| \le L),\\
    0 &(|z| > L).
    \end{cases}
\end{align}
Also, we put
\begin{align}\label{SigmaL}
    \Sigma_L
    = \left\{ x \in \Omega \mid
    L-1 \le |z| \le L
    \right\}.
\end{align}
Note that
$\supp \partial_z \varphi_L \subset \Sigma_L$.

Finally, we prove a Poincar\'{e}-type inequality for $r$-direction in $\Omega$ and $\Sigma_L$,
which will be used for
$\partial_z v$ and $\partial_{\theta} v$.
%%%%%%%%%%%%%%%%%%%%%%%%%%%%%%%%%%%%%%%%%%%%%%
\begin{lemma}\label{lem:Poincare}
Let $f = f(r,\theta,z)$ is a smooth function on $\overline{\Omega}$
satisfying the boundary condition
$f(R_j, \theta, z) = 0 \ (j=1,2)$.
Let $L > 1$.
Then, we have
\begin{align}
    \| f \sqrt{\varphi_L} \|_{L^2(D)}
    \le 
    \sqrt{C_P} \| \partial_r f \sqrt{\varphi_L} \|_{L^2(D)},
\end{align}
where
$D$ denotes $\Omega$ or $\Sigma_L$,
and
\begin{align}\label{eq:Poincare:const}
    C_P = \frac{R_2(R_2-R_1)^2}{R_1 \pi^2}.
\end{align}
\end{lemma}
%%%%%%%%%%%%%%%%%%%%%%%%%%%%%%%%%%%%%%%%%%%%%%
\begin{proof}
When $D = \Omega$,
using the cylindrical coordinates and applying the Poincar\'{e} inequality in $r$-direction,
we calculate
\begin{align}
    \| f \sqrt{\varphi_L} \|_{L^2(\Omega)}^2
    &=
    \int_{\re} \int_0^{2\pi} \| f \sqrt{r} \|_{L^2(R_1,R_2)}^2 \varphi_L(z) \,d\theta dz \\
    &\le
    R_2 \int_{\re} \int_0^{2\pi} \| f \|_{L^2(R_1,R_2)}^2
    \varphi_L(z) \,d\theta dz \\
    &\le
    R_2 \frac{(R_2-R_1)^2}{\pi^2}
    \int_{\re} \int_0^{2\pi} \| \partial_r f\|_{L^2(R_1,R_2)}^2 \varphi_L(z) \,d\theta dz \\
    &\le
    \frac{R_2(R_2-R_1)^2}{R_1 \pi^2}
    \int_{\re} \int_0^{2\pi} \| \partial_r f \sqrt{r} \|_{L^2(R_1,R_2)}^2 \varphi_L(z) \,d\theta dz \\
    &=
    C_P
    \| \partial_r f \sqrt{\varphi_L} \|_{L^2(\Omega)}^2.
\end{align}
The case $D = \Sigma_L$ can be proved in the completely same way.
\end{proof}

%%%%%%%%%%%%%%%%%%%%%%%%%%%%%%%%%%%%%%%%%%%%%%%%%%%%%%
%%%%%%%%%%%%%%%%%%%%%%%%%%%%%%%%%%%%%%%%%%%%%%%%%%%%%%
\section{Proof of Theorem \ref{thm:liouville}}
Let us prove Theorem \ref{thm:liouville}.
Assume that
$(v,p)$
is an axially symmetric smooth solution of \eqref{ns} in $\Omega$.
Following the argument of
\cite{BaGuWaXi},
we will first show
\begin{align}\label{eq:dzv0}
    \partial_z v \equiv 0.
\end{align}
To this end,
we differentiate the equations \eqref{ns}
with respect to $z$
and obtain
\begin{align}
\label{ns:dzv}
    \left\{\begin{alignedat}{3}
    \left( \partial_z v^r \partial_r + \partial_z v^z \partial_z \right) v^r
    + \left( v^r \partial_r + v^z \partial_z \right) \partial_z v^r
    - \frac{2v^{\theta} \partial_z v^{\theta}}{r} + \partial_z \partial_r p
    &=
    \nu \left( \partial_r^2 + \frac{1}{r} \partial_r
    + \partial_z^2 - \frac{1}{r^2} \right) \partial_z v^r,\\
    \left( \partial_z v^r \partial_r + \partial_z v^z \partial_z \right) v^{\theta}
    + \left( v^r \partial_r + v^z \partial_z \right) \partial_z v^{\theta}
    + \frac{v^{\theta} \partial_z v^r + v^r \partial_z v^{\theta} }{r}
    &=
    \nu \left( \partial_r^2 + \frac{1}{r} \partial_r
    + \partial_z^2 - \frac{1}{r^2} \right) \partial_z v^{\theta},\\
    \left( \partial_z v^r \partial_r + \partial_z v^z \partial_z \right) v^z
    + \left( v^r \partial_r + v^z \partial_z \right) \partial_z v^z
    + \partial_z^2 p
    &=
    \nu \left( \partial_r^2 + \frac{1}{r} \partial_r + \partial_z^2 \right) \partial_z v^z,\\
    \partial_z \partial_r v^r + \frac{\partial_z v^r}{r} + \partial_z^2 v^z
    &= 0.
    \end{alignedat}\right.
\end{align}
Moreover, we have the boundary conditions for
$\partial_z v$:
\begin{align}\label{bc:pzv}
    \partial_z v (R_j, z) = 0 \quad (j=1,2).
\end{align}
Let
$L > 1$
and take the test function
$\varphi_L$
and the region
$\Sigma_L$
defined by \eqref{varphiL} and \eqref{SigmaL}, respectively.
We multiply the equations of 
$v^r, v^{\theta}, v^z$ in \eqref{ns:dzv}
by
$\partial_z v^r \varphi_L(z)$,
$\partial_z v^{\theta} \varphi_L(z)$,
$\partial_z v^z \varphi_L(z)$,
respectively,
and sum up them and integrate them over $\Omega$.
Then, we have the integral identity
\begin{align}\label{eq:I-V}
    I &= I\!I + I\!I\!I + I\!V + V.
\end{align}
Here, $I = I^r + I^{\theta} + I^z$ is the sum related to the viscous terms 
of the R.H.S. of (\ref{ns:dzv}) defined by
\begin{align}
    I^{\lambda}
    &=
        \nu \int_{\Omega} \left( \partial_r^2 + \frac{1}{r} \partial_r
        + \partial_z^2 - \frac{1}{r^2} \right) \partial_z v^{\lambda} \partial_z v^{\lambda} \varphi_L(z) \,dx
        \quad (\lambda = r,\theta),\\
    I^z 
    &=
        \nu \int_{\Omega} \left( \partial_r^2 + \frac{1}{r} \partial_r
        + \partial_z^2 \right) \partial_z v^{z} \partial_z v^z \varphi_L(z) \,dx.
\end{align}
The terms $I\!I$, $I\!I\!I$, and $I\!V$ are related to the nonlinear terms 
of the L.H.S. of (\ref{ns:dzv}) defined by 
\begin{align}
    I\!I
    &=
    \sum_{\lambda=r,\theta,z} I\!I^{\lambda}
    =
    \sum_{\lambda=r,\theta,z}
        \int_{\Omega} 
        \left( \partial_z v^r \partial_r + \partial_z v^z \partial_z \right)
        v^{\lambda} \partial_z v^{\lambda} \varphi_L(z) \,dx, \\
    I\!I\!I
    &=
    \sum_{\lambda=r,\theta,z} I\!I\!I^{\lambda}
    =
    \sum_{\lambda=r,\theta,z}
    \int_{\Omega} \left( v^r \partial_r + v^z \partial_z \right) 
    \partial_z v^{\lambda} \partial_z v^{\lambda} \varphi_L(z) \,dx,\\
    I\!V
    &=
    -2 \int_{\Omega}
    \frac{v^{\theta} \partial_z v^{\theta}}{r} \partial_z v^r \varphi_L(z) \,dx
    +
    \int_{\Omega}
    \frac{v^{\theta} \partial_z v^r + v^r \partial_z v^{\theta} }{r} \partial_z v^{\theta} \varphi_L(z) \,dx \\
    &=
    \int_{\Omega} \frac{1}{r} \left( v^r \partial_z v^{\theta} - v^{\theta} \partial_z v^r \right)
        \partial_z v^{\theta} \varphi_L(z) \,dx,
\end{align}
respectively.
Finally, the term $V$ is the sum related to the pressure terms of (\ref{ns:dzv}) 
defied by 
\begin{align}
    V
    &=
    \int_{\Omega} \left(
    \partial_z\partial_r p \partial_z v^r
    + \partial_z^2 p \partial_z v^z
    \right) \varphi_L(z) \,dx.
\end{align}

%%%% I
First, we compute the viscous terms $I$.
For $\lambda = r, \theta$,
the integration by parts
with noting
$r(\partial_r^2 + \frac{1}{r}\partial_r) = \partial_r (r \partial_r)$
implies
\begin{align}
    I^{\lambda} 
    &= \nu \int_{\Omega} \left( \partial_r^2 + \frac{1}{r} \partial_r
    + \partial_z^2 - \frac{1}{r^2} \right) \partial_z v^{\lambda} \partial_z v^{\lambda} \varphi_L(z) \,dx \\
    &=
    - 2\pi \nu \int_{\re} \int_{R_1}^{R_2}
    \left(
        |\partial_r \partial_z v^{\lambda} |^2
        + | \partial_z^2 v^{\lambda} |^2
        + \frac{|\partial_z v^{\lambda}|^2}{r^2}
    \right) \varphi_L(z) r \,drdz \\
    &\quad
    - 2\pi \nu \int_{\re} \int_{R_1}^{R_2}
        \partial_z^2 v^{\lambda} 
        \partial_z v^{\lambda} \partial_z \varphi_L(z) r \,drdz \\
    &=: I_1^{\lambda} + I_2^{\lambda}.
\end{align}
In the same way, for $I^z$, we have
\begin{align}
    I^z
    &=
    \nu \int_{\Omega} \left( \partial_r^2 + \frac{1}{r} \partial_r
    + \partial_z^2 \right) \partial_z v^{z}
    \partial_z v^z \varphi_L(z) \,dx \\
    &=
    - 2\pi
    \nu \int_{\re} \int_{R_1}^{R_2}
    \left(
        |\partial_r \partial_z v^{z} |^2
        + | \partial_z^2 v^{z} |^2
    \right) \varphi_L(z) r\,drdz \\
    &\quad
    - 2\pi \nu
    \int_{\re} \int_{R_1}^{R_2}
     \partial_z^2 v^{z}
        \partial_z v^{z}  \partial_z \varphi_L(z) r \,drdz \\
    &=: I_1^z + I_2^z.
\end{align}
Now, for later purpose, we express the sum of good terms by $Y(L)$:
\begin{align}\label{eq:def:YL}
    Y(L)
    &:=
    - (I_1^{r} + I_1^{\theta} + I_1^z) \\
    &=
    \nu \left(
    \| \partial_r \partial_z v \sqrt{\varphi_L} \|_{L^2(\Omega)}^2
    + \| \partial_z^2 v \sqrt{\varphi_L} \|_{L^2(\Omega)}^2
    + \sum_{\lambda=r,\theta} \| r^{-1} \partial_z v^{\lambda} \sqrt{\varphi_L} \|_{L^2(\Omega)}^2
    \right).
\end{align}
We remark that the definition of $\varphi_L(z)$ (see \eqref{varphiL}) implies
\begin{align}
    Y'(L)
    &=
     \nu \left(
    \| \partial_r \partial_z v \|_{L^2(\Sigma_L)}^2
    + \| \partial_z^2 v \|_{L^2(\Sigma_L)}^2
    + \sum_{\lambda=r,\theta} \| r^{-1} \partial_z v^{\lambda} \|_{L^2(\Sigma_L)}^2
    \right).
\end{align}
By using this $Y(L)$, the term $I$ can be written as
\begin{align}\label{eq:I:Y}
    I = - Y(L) + \sum_{\lambda=r,\theta,z} I_2^{\lambda}.
\end{align}
Let us estimate the remainder terms $I_2^{\lambda}$ for $\lambda = r,\theta,z$.
Since $\partial_z v^{\lambda} \in L^{\infty}(\Omega)$
by Lemma \ref{lem:bdd} and
$|\Sigma_L| \le C$
with some constant $C$ independent of $L$,
we estimate
\begin{align}\label{eq:est:lin}
    |I_2^{\lambda}|
    &\le 
    C \| \partial_z^2 v^{\lambda} \|_{L^2(\Sigma_L)}
    \| \partial_z v^{\lambda} \|_{L^2(\Sigma_L)}
    \le
    C \| \partial_z^2 v \|_{L^2(\Sigma_L)} 
    \le C \sqrt{Y'(L)},
\end{align}
where the constant $C$ is independent of $L$.

%%% II 
Next,
we consider the nonlinear term $I\!I$.
By the integration by parts
and the divergence free condition
$\partial_r (r \partial_z v^r) + \partial_z (r \partial_z v^z) = 0$,
the terms
$I\!I^{\lambda}$ for $\lambda = r,\theta,z$
are written as
\begin{align}
    I\!I^{\lambda}
    &=
    - 2\pi \int_{\re} \int_{R_1}^{R_2}
    \left(
    \partial_z v^r \partial_r + \partial_z v^{z} \partial_z 
    \right) \partial_z v^{\lambda} v^{\lambda} \varphi_L(z) r \,drdz \\
    &\quad - 2\pi\int_{\re} \int_{R_1}^{R_2}
    \partial_z v^{z} 
    v^{\lambda} \partial_z v^{\lambda } \partial_z \varphi_L(z) r \,drdz \\
    &=: I\!I_1^{\lambda} + I\!I_2^{\lambda}.
\end{align}
The H\"{o}lder inequality implies
\begin{align}
    |I\!I^{\lambda}_{1}|
    &\le
    \| \partial_z v^r \sqrt{\varphi_L} \|_{L^2(\Omega)}
    \| \partial_r \partial_z v^{\lambda} \sqrt{\varphi_L} \|_{L^2(\Omega)}
    \| v \|_{L^{\infty}(\Omega)} \\
    &\quad 
    + \| \partial_z v^z \sqrt{\varphi_L} \|_{L^2(\Omega)}
    \| \partial_z^2 v^{\lambda} \sqrt{\varphi_L} \|_{L^2(\Omega)}
    \| v \|_{L^{\infty}(\Omega)}
\end{align}
Moreover, by noting the boundary condition \eqref{bc:pzv}
and applying Lemma \ref{lem:Poincare},
we further estimate
\begin{align}
    |I\!I_1^{\lambda}|
    &\le
    \| v \|_{L^{\infty}(\Omega)} \sqrt{C_P}
    \left( 
    \| \partial_r \partial_z v^r \sqrt{\varphi_L} \|_{L^2(\Omega)} 
    \| \partial_r \partial_z v^{\lambda} \sqrt{\varphi_L} \|_{L^2(\Omega)} 
    \right. \\
    &\qquad\qquad\qquad\qquad
    + \left.
    \| \partial_r\partial_z v^z \sqrt{\varphi_L} \|_{L^2(\Omega)}
    \| \partial_z^2 v^{\lambda} \sqrt{\varphi_L} \|_{L^2(\Omega)}
    \right).
\end{align}
Hence, to the term
$I\!I_1 := I\!I_1^{r} + I\!I_1^{\theta} + I\!I_1^z$,
we apply the Schwarz inequality to conclude
\begin{align} \label{eq:est:I}
    & |I\!I_1|  \\
    &\le
    \| v \|_{L^{\infty}(\Omega)} \sqrt{C_P}
   \left(
    2 \| \partial_r \partial_z v^r \sqrt{\varphi_L} \|_{L^2(\Omega)}^2
    + \frac{1}{2} \sum_{\lambda=r,\theta,z}
        \| \partial_r\partial_z v^{\lambda} \sqrt{\varphi_L} \|_{L^2(\Omega)}^2
        \right. \\ 
    & \mbox{ } 
    \left.
    + 2 \| \partial_r \partial_z v^z \sqrt{\varphi_L} \|_{L^2(\Omega)}^2
    + \frac{1}{2} \sum_{\lambda=r,\theta,z}
        \| \partial_z^2 v^{\lambda} \sqrt{\varphi_L} \|_{L^2(\Omega)}^2
   \right). 
%\label{eq:est:I}
\end{align}
Note that the terms in the parentheses are members of $Y(L)$.
Similarly to (\ref{eq:est:lin}) we have by Lemma  \ref{lem:bdd} that 
$$
\|\partial_zv\|_{L^2(\Sigma_L)} \le \|\partial_zv\|_{L^\infty(\Omega)}|\Sigma_L|^{\frac12} 
\le C 
$$
with the constant $C$ independent of $L$.  
Hence, applying  
the Poinc\'are inequality in $\Sigma_L$ with the aid of (\ref{bc:pzv}) 
to the estimate for the term
$I\!I_2 := I\!I_2^r + I\!I_2^{\theta} + I\!I_2^z$, 
we have by the H\"{o}lder inequality that
\begin{align}\label{eq:est:II}
    |I\!I_2|
    &\le
    C \| \partial_z v \|_{L^2(\Sigma_L)}^2 \| v \|_{L^{\infty}(\Omega)}
    \le 
    C \| \partial_r \partial_z v \|_{L^{2}(\Sigma_L)}
    \le C\sqrt{Y'(L)}.
\end{align}

%%%%%%%%%%%%%%% III
Let us estimate the term $I\!I\!I$.
For $\lambda = r, \theta, z$,
we write
$I\!I\!I^{\lambda}$
as
\begin{align}
    I\!I\!I^{\lambda}
    &=
    2\pi
    \int_{\re} \int_{R_1}^{R_2}
    \left( v^r \partial_r + v^z \partial_z \right) 
    \left( \frac{1}{2} |\partial_z v^{\lambda} |^2 \right)
    \varphi_L(z) r \,drdz.
\end{align}
Since divergence free condition means that 
$\partial_r (r v^r) + \partial_z (r v^z) = 0$, 
we have by integration by parts with the aid of (\ref{bc:pzv}) that 
\begin{align}
    I\!I\!I^{\lambda}
    =
    - \pi \int_{\re} \int_{R_1}^{R_2}
    v^z \partial_z \varphi_L(z) |\partial_z v^{\lambda} |^2 r \,drdz.
\end{align}
Then, in the same way to the term $I\!I_2$,
we obtain
\begin{align}\label{eq:est:III}
    |I\!I\!I|
    &\le
    C \| \partial_r \partial_z v \|_{L^{2}(\Sigma_L)}
    \le
    C \sqrt{Y'(L)}.
\end{align}

%%%%% IV
The remaining nonlinear term $I\!V$
can be treated similarly to $I$.
Indeed, using the H\"{o}lder inequality
and Lemma \ref{lem:Poincare},
we have
\begin{align}\label{eq:est:IV}
    \qquad
    |I\!V|
    &\le 
    \| r^{-1} \partial_z v^{r} \sqrt{\varphi_L} \|_{L^2(\Omega)}
    \| \partial_z v^{\theta} \sqrt{\varphi_L} \|_{L^2(\Omega)}
    \| v \|_{L^{\infty}(\Omega)} \\
    &\quad 
    + \| r^{-1} \partial_z v^{\theta} \sqrt{\varphi_L} \|_{L^2(\Omega)}
    \| \partial_z v^{\theta} \sqrt{\varphi_L} \|_{L^2(\Omega)}
    \| v \|_{L^{\infty}(\Omega)} \\
    &\le 
    \| v \|_{L^{\infty}(\Omega)} \sqrt{C_P}
    \left( 
    \| r^{-1} \partial_z v^{r} \sqrt{\varphi_L} \|_{L^2(\Omega)}
    +  \| r^{-1} \partial_z v^{\theta} \sqrt{\varphi_L} \|_{L^2(\Omega)}
    \right) \\
    &\quad \times
    \| \partial_r \partial_z v^{\theta} \sqrt{\varphi_L} \|_{L^2(\Omega)} \\
    &\le
    \| v \|_{L^{\infty}(\Omega)}
    \sqrt{C_P} \\
    &\quad \times 
    \left(
    \frac{1}{2} \sum_{\lambda=r,\theta}
    \| r^{-1} \partial_z v^{\lambda} \sqrt{\varphi_L} \|_{L^2(\Omega)}^2
    + \| \partial_r \partial_z v^{\theta} \sqrt{\varphi_L} \|_{L^2(\Omega)}^2
    \right).
\end{align}
We again note that the terms in the parentheses are members of $Y(L)$.

%%%% V
Finally, we estimate the pressure term $V$.
Since the divergence free condition yields
$\partial_r (r \partial_z v^r) + \partial_z (r\partial_z v^z) = 0$,
we have by integration by parts with (\ref{bc:pzv}) that 
\begin{align}
    V
    &=
    2\pi \int_{\re}\int_{R_1}^{R_2}
    \left(
    \partial_z\partial_r p \partial_z v^r
    + \partial_z^2 p \partial_z v^z
    \right) \varphi_L(z) r \,dr dz \\
    &=
    - \int_{\re} \int_{R_1}
    \partial_z p
    \partial_z v^z \partial_z \varphi_L (z) r \,drdz.
\end{align}
By
$\partial_z p \in L^{\infty}(\Omega)$,
which follows from Lemma \ref{lem:bdd},
$|\Sigma_L| \le C$
with some constant $C$ independent of $L$,
and Lemma \ref{lem:Poincare},
we obtain
\begin{align}\label{eq:est:V}
    |V|
    &\le
    C \| \partial_z p \|_{L^2(\Sigma_L)}
    \| \partial_z v \|_{L^2(\Sigma_L)}
    \le
    C \| \partial_r \partial_z v \|_{L^2(\Sigma_L)}
    \le
    C \sqrt{Y'(L)}.
\end{align}

Putting the estimates
\eqref{eq:I:Y},
\eqref{eq:est:lin}, \eqref{eq:est:I}, \eqref{eq:est:II}, \eqref{eq:est:III}, \eqref{eq:est:IV}, and \eqref{eq:est:V}
together into the original integral identity \eqref{eq:I-V},
we have
\begin{align}
    Y(L)
    &\le
    \frac{2\sqrt{C_P}}{\nu}
    \| v \|_{L^{\infty}(\Omega)}
    Y(L)
    + 
    C \sqrt{Y'(L)}.
\end{align}
Therefore, putting
$C_1(\nu,R_1,R_2) := \frac{\nu}{2\sqrt{C_P}}$
and using the assumption \eqref{eq:ass:v:Linf}
on $\|v \|_{L^{\infty}(\Omega)}$,
we see that the first term of the right-hand side
can be absorbed to the left-hand side.
Hence, we conclude
\begin{align}\label{eq:diffineq:Y}
    Y(L)
    \le
    C \sqrt{Y'(L)}.
\end{align}
Note that the constant $C$
in the right-hand side is independent of
$L$.

%%% 
The differential inequality \eqref{eq:diffineq:Y}
enables us to reach the first goal \eqref{eq:dzv0}.
Indeed, by \eqref{eq:diffineq:Y} it holds that 
$$
Y(L) = 0 
\quad
\mbox{for all $L>1$}.  
$$
Suppose the contrary.  Then  there exists some $L_0 > 1$ such that 
$Y(L_0) > 0$.
Since $Y(L)$ is a non-decreasing function of $L$,
we have $Y(L) > 0$ for all $L \ge L_0$.
Thus, from \eqref{eq:diffineq:Y}, we deduce
for $L > L_0$ that
\begin{align}
    1 \le C Y(L)^{-2} Y'(L) = C \left( - Y(L)^{-1} \right)'.
\end{align}
Integrating it over
$[L_0,L]$
leads to 
\begin{align}
    L-L_0 \le C \left( -Y(L)^{-1} + Y(L_0)^{-1} \right)
    \le C Y(L_0)^{-1}.
\end{align}
However, letting $L$ sufficiently large,
we reach contradiction to conclude that $Y(L) = 0$ for all $L > 1$. 
Now, we have that 
$\nabla_{r,z} \partial_z v \equiv 0$, 
which means that the function
$\partial_z v$
is a constant vector.
Combining this with the boundary condition
\eqref{bc:pzv} implies that
$\partial_z v \equiv 0$.
Thus, we have \eqref{eq:dzv0}.

%%%
Finally, we show that the solutions
$(v,p)$
has the form described in the statement of the theorem.
First, by
$\partial_z v \equiv 0$
and
the divergence free condition, we have
\begin{align}
    \partial_r v^r + \frac{v^r}{r}
    = \frac{1}{r} \partial_r (r v^r) = 0,
\end{align}
which shows
$\partial_r (r v^r) \equiv 0$,
that is,
$r v^r$
is a constant.
However, the boundary condition
$v^r = 0$ on $\partial \Omega$
again implies
$v^r \equiv 0$.

Going back to the system \eqref{ns:dzv},
we have
\begin{align}
    \partial_z \partial_r p = \partial_z^2 p = 0,
\end{align}
which implies that
$\partial_z p = a$
with some constant $a \in \re$.
Integrating it gives
\begin{align}\label{eq:p}
    p(r,z) = a z + h(r)
\end{align}
with some smooth function $h(r)$.

We further go back to the original system \eqref{ns}
and determine
$v^{\theta}, v^{z}$, and $h(r)$.
First,
by noting that $v^r = 0$ and $v$ is independent of $z$,
the second equation of \eqref{ns}
yields that $v^{\theta}$ is subject to the equation
\begin{align}
    \left( \partial_r^2 + \frac{1}{r} \partial_r - \frac{1}{r^2} \right) v^{\theta} = 0.
\end{align}
This is the Euler--Cauchy equation and
we find the general solution of the form
\begin{align}
    v^{\theta} = Ar + \frac{B}{r}.
\end{align}
The boundary condition gives
\begin{align}
    R_2\omega_2 &= v^{\theta}(R_2) = AR_2 + \frac{B}{R_2},\\
    R_1\omega_1 &= v^{\theta} (R_1) = AR_1+ \frac{B}{R_1}.
\end{align}
Solving this, we determine the constants
$A, B$ and obtain
\begin{align}\label{eq:vtheta}
    v^{\theta} (r) =
     \frac{R_2^2\omega_2 - R_1^2\omega_1}{R_2^2-R_1^2} r
    + \frac{R_1^2R_2^2(-\omega_2+\omega_1)}{R_2^2-R_1^2} \frac{1}{r}.
\end{align}
Next, from the third equation of \eqref{ns}
and the formula \eqref{eq:p},
we have the equation of $v^{z}$:
\begin{align}\label{eq:vz}
    \nu \left( \partial_r^2 + \frac{1}{r} \partial_r \right)
    v^z = a,
\end{align}
that is,
\begin{align}
    \frac{1}{r} \partial_r \left( r \partial_r v^z \right) = \frac{a}{\nu}.
\end{align}
This implies
\begin{align}
    r \partial_r v^z(r) = D + \frac{a}{2\nu}r^2
\end{align}
with some constant $D$.
Integrating it over $[R_1, r]$ and using the
boundary condition
$v^z(R_1) = 0$ by \eqref{eq:bc},
we have
\begin{align}
    v^z(r) = D\log \frac{r}{R_1} + \frac{a}{4\nu}(r^2-R_1^2).
\end{align}
From the boundary condition
$v^z(R_2)=0$ by \eqref{eq:bc},
the constant $D$ is determined as
\begin{align}
    D = - \frac{a}{4\nu} \frac{R_2^2-R_1^2}{\log R_2/R_1}.
\end{align}
Thus, we conclude
\begin{align}\label{eq:vz}
    v^{z}(r) = 
    \frac{a}{4\nu}
    \left[ (r^2 - R_1^2) - \frac{R_2^2-R_1^2}{\log (R_2/R_1)} \log \frac{r}{R_1} \right].
\end{align}

Finally, from the first equation of \eqref{ns},
we deduce
\begin{align}
    - \frac{(v^{\theta})^2}{r} + \partial_r p = 0.
\end{align}
This and the formulas \eqref{eq:p} and \eqref{eq:vtheta} lead to
\begin{align}
    h'(r)
    &=
    \frac{1}{r}
    \left[
    \frac{R_2^2\omega_2-R_1^2\omega_1}{R_2^2-R_1^2} r
    + \frac{R_1^2R_2^2(-\omega_2+\omega_1)}{R_2^2-R_1^2} \frac{1}{r}
    \right]^2 \\
    &=
    \left( \frac{R_2^2\omega_2-R_1^2\omega_1}{R_2^2-R_1^2} \right)^2 r
    + \frac{2 R_1^2 R_2^2(-\omega_2+\omega_1)(R_2^2\omega_2-R_1\omega_1)}{(R_2^2-R_1^2)^2} \frac{1}{r}
    + \left( \frac{R_1^2R_2^2(-\omega_2+\omega_1)}{R_2^2-R_1^2}  \right)^2 \frac{1}{r^3}.
\end{align}
Integrating it, we have
\begin{align}
    h(r)
    &=
    b +
    \frac{1}{2} \left( \frac{R_2^2\omega_2-R_1^2\omega_1}{R_2^2-R_1^2} \right)^2 r^2
    + \frac{2 R_1^2 R_2^2(-\omega_2+\omega_1)(R_2^2\omega_2-R_1^2\omega_1)}{(R_2^2-R_1^2)^2} \log r \\
    &\quad 
    -\frac{1}{2} \left( \frac{R_1^2R_2^2(-\omega_2+\omega_1)}{R_2^2-R_1^2}  \right)^2 \frac{1}{r^2}
\end{align}
with some constant
$b \in \re$,
that is, the pressure is given by
\begin{align}\label{eq:p:conclusion}
    p(r,z) 
    &=
    az + b +
    \frac{1}{2} \left( \frac{R_2^2\omega_2-R_1^2\omega_1}{R_2^2-R_1^2} \right)^2 r^2
    + \frac{2 R_1^2 R_2^2(-\omega_2+\omega_1)(R_2^2\omega_2-R_1^2\omega_1)}{(R_2^2-R_1^2)^2} \log r \\
    &\quad
    -\frac{1}{2} \left( \frac{R_1^2R_2^2(-\omega_2+\omega_1)}{R_2^2-R_1^2}  \right)^2 \frac{1}{r^2}.
\end{align}
Rewriting \eqref{eq:vtheta}, \eqref{eq:vz}, and \eqref{eq:p:conclusion} by using
$\mu$ and $\eta$ defined by \eqref{eq:mu:eta}
completes the proof of Theorem \ref{thm:liouville}.

%%%%%%%%%%%%%%%%%%%%%%%%%%%%%%%%
%%%%%%%%%%%%%%%%%%%%%%%%%%%%%%%
%%%%%%%%%%%%%%%%%%%%%%%%%%%%%%%
\section{Proof of Theorem \ref{thm:liouville:2}}
Let us prove Theorem \ref{thm:liouville:2}.
Assume that $(v,p)$
is a smooth solution to \eqref{ns2}.
We differentiate the equations \eqref{ns2} with respect to 
$\theta$
to obtain
\begin{align}
\label{ns2:dtheta}
    \left\{\begin{alignedat}{3}
    \left( \partial_{\theta} v\cdot \nabla \right) v^r
    +\left( v\cdot \nabla \right) \partial_{\theta} v^r
    - \frac{2v^{\theta} \partial_{\theta}v^{\theta}}{r} + \partial_{\theta} \partial_r p
    &=
    \nu \left( \Delta - \frac{1}{r^2} \right) \partial_{\theta} v^r
    - \nu \frac{2}{r^2} \partial_{\theta}^2 v^{\theta},\\
    \left( \partial_{\theta} v\cdot \nabla \right) v^{\theta}
    + \left( v\cdot \nabla \right) \partial_{\theta} v^{\theta}
    + \frac{v^r \partial_{\theta} v^{\theta} + v^{\theta} \partial_{\theta} v^r}{r}
    + \frac{1}{r} \partial_{\theta}^2 p
    &=
    \nu \left( \Delta - \frac{1}{r^2} \right) \partial_{\theta} v^{\theta}
    + \nu \frac{2}{r^2} \partial_{\theta}^2 v^r,\\
    \left( \partial_{\theta} v\cdot \nabla \right) v^z
    + \left( v\cdot \nabla \right) \partial_{\theta} v^z
    + \partial_{\theta} \partial_z p
    &=
    \nu \Delta \partial_{\theta} v^z,\\
    \ \frac{1}{r} \partial_r (r \partial_{\theta} v^r)
    + \frac{1}{r} \partial_{\theta}^2 v^{\theta} 
    + \partial_{z} \partial_\theta v^z
    &= 0.
    \end{alignedat}\right.
\end{align}
Here, we recall that the operators
$(v\cdot \nabla)$ and $\Delta$ are
defined by
\eqref{eq:nabla:rtz} and \eqref{eq:Delta:rtz}, respectively.
We also have the boundary conditions
\begin{align}
    \partial_{\theta} v (R_j, \theta, z) = 0
    \quad (j=1,2).
\end{align}
Let $L>1$
and take the test function
$\varphi_L(z)$
and the region
$\Sigma_L$
defined by \eqref{varphiL} and \eqref{SigmaL},
respectively.
Similarly to the previous section,
we multiply the equations of
$v^r, v^{\theta}, v^z$
in \eqref{ns2:dtheta}
by
$\partial_{\theta} v^r \varphi_L(z),
\partial_{\theta} v^{\theta} \varphi_L(z),
\partial_{\theta} v^z \varphi_L(z)$,
respectively,
and sum up and integrate them over $\Omega$.
As a result, we have the integral identity
\begin{align}\label{eq:4:I-V}
    I = I\!I + I\!I\!I + I\!V + V.
\end{align}
Here, $I = I^r + I^{\theta} + I^z$
is the sum of the viscous term defined by
\begin{align}
    I^r
    &=
    \nu \int_{\Omega} \left[
    \left( \Delta - \frac{1}{r^2} \right) \partial_{\theta} v^r \partial_{\theta} v^r
    - \frac{2}{r^2} \partial_{\theta}^2 v^{\theta} \partial_{\theta} v^r \right] \varphi_{L}(z) \,dx,\\
    I^{\theta}
    &=
    \nu \int_{\Omega} \left[
    \left( \Delta - \frac{1}{r^2} \right) \partial_{\theta} v^{\theta} \partial_{\theta} v^{\theta}
    + \frac{2}{r^2} \partial_{\theta}^2 v^r \partial_{\theta} v^{\theta} 
    \right] \varphi_{L}(z) \,dx,\\
    I^z
    &=
    \nu \int_{\Omega} \Delta \partial_{\theta} v^z \partial_{\theta} v^z \varphi_L(z) \,dx.
\end{align}
Also, $I\!I$, $I\!I\!I$, and $I\!V$
are the nonlinear terms defined by
\begin{align}
    I\!I
    &=
    \sum_{\lambda=r,\theta,z} I\!I^{\lambda}
    =
    \sum_{\lambda=r,\theta,z}
    \int_{\Omega}
    \left( \partial_{\theta} v \cdot \nabla \right) v^{\lambda}
    \partial_{\theta} v^{\lambda} \varphi_L(z)
    \,dx,\\
    I\!I\!I
    &=
    \sum_{\lambda=r,\theta,z} I\!I\!I^{\lambda}
    =
    \sum_{\lambda=r,\theta,z}
    \int_{\Omega} (v\cdot \nabla )\partial_{\theta} 
    v^{\lambda} \partial_{\theta} v^{\lambda} \varphi_L(z) \,dx, \\
    I\!V
    &=
    \int_{\Omega} \left( - \frac{2v^{\theta}\partial_{\theta} v^{\theta} }{r} \partial_{\theta}v^r
    + \frac{v^r\partial_{\theta} v^{\theta} + v^{\theta} \partial_{\theta} v^r}{r} \partial_{\theta} v^{\theta} \right) \varphi_L(z)\,dx \\
    &=
    \int_{\Omega}
    \frac{1}{r} \left( v^r \partial_{\theta} v^{\theta} - v^{\theta} \partial_{\theta} v^r\right)
    \partial_{\theta} v^{\theta} \varphi_L(z) \,dx.
\end{align}
Finally, $V$ is the sum of the pressure terms defined by
\begin{align}
    V
    &=
    \int_{\Omega} \left( \partial_{\theta} \partial_r p \partial_{\theta} v^r
    + \frac{1}{r} \partial_{\theta}^2 p \partial_{\theta} v^{\theta}
    + \partial_{\theta} \partial_z p \partial_{\theta} v^z
    \right) \varphi_L(z) \,dx.
\end{align}

%%% I 
First, we consider the term $I$.
By the integration by parts with the aid of the boundary condition 
$\pt_\theta v(R_j, \theta, z)=0$ for $j=1, 2$, we infer that
\begin{align}
\label{eq:visc:rt}
    I^r + I^{\theta}
    &=
    - \nu \int_{\Omega} 
    \left[
    |\partial_r \partial_{\theta} v^r|^2
    + \frac{1}{r^2} |\partial_{\theta}^2 v^r|^2
    + |\partial_z \partial_{\theta} v^r|^2
    + \frac{1}{r^2} |\partial_{\theta} v^r|^2
    + \frac{2}{r^2} \partial_{\theta}^2 v^{\theta} \partial_{\theta} v^r \right.\\
    &\quad \left.
    + |\partial_r \partial_{\theta} v^{\theta} |^2
    + \frac{1}{r^2} |\partial_{\theta}^2 v^{\theta} |^2
    + |\partial_z \partial_{\theta} v^{\theta}|^2
    + \frac{1}{r^2} |\partial_{\theta} v^{\theta} |^2
    - \frac{2}{r^2} \partial_{\theta}^2 v^r \partial_{\theta} v^{\theta}
    \right] \varphi_L(z) \,dx \\
    &\quad
    -\nu \int_{\Omega}
    \left(
    \partial_z \partial_{\theta} v^r
    \partial_{\theta} v^r
    + \partial_z \partial_{\theta} v^{\theta} \partial_{\theta} v^{\theta} 
    \right) \partial_z\varphi_L(z) \,dx \\
    &=
    -\nu \int_{\Omega} \left[
    |\partial_r \partial_{\theta} v^r|^2
    + |\partial_z \partial_{\theta} v^r |^2
    + |\partial_r \partial_{\theta} v^{\theta} |^2
    + |\partial_z \partial_{\theta} v^{\theta} |^2 
    \right. \\
    &\quad \left.
    + \frac{1}{r^2} |\partial_{\theta}^2 v^{\theta} + \partial_{\theta} v^r |^2
    + \frac{1}{r^2} |\partial_{\theta}^2 v^r - \partial_{\theta} v^{\theta} |^2
    \right] \varphi_L(z) \,dx \\
    &\quad
    -\nu \int_{\Omega}
    \left(
    \partial_z \partial_{\theta} v^r
    \partial_{\theta} v^r
    + \partial_z \partial_{\theta} v^{\theta} \partial_{\theta} v^{\theta} 
    \right) \partial_z\varphi_L(z) \,dx \\
    &=:
    I_1^{r,\theta} + I_2^{r,\theta}
\end{align}
and
\begin{align}
\label{eq:visc:z}
    I^z
    &=
    \nu \int_{\Omega} \Delta \partial_{\theta} v^z \partial_{\theta} v^z \varphi_L(z) \,dx \\
    &=
    -\nu \int_{\Omega} \left[ 
    | \partial_r \partial_{\theta} v^z |^2
    + \frac{1}{r^2} | \partial_{\theta}^2 v^z |^2
    + |\partial_z \partial_{\theta} v^z |^2 
    \right] \varphi_L(z) \,dx \\
    &\quad 
    - \nu \int_{\Omega} 
    \partial_z \partial_{\theta} v^z
    \partial_{\theta} v^z \partial_z\varphi_L(z) \,dx \\
    &=:
    I_1^z + I_2^z.
\end{align}
Similarly to the previous section,
we define
\begin{align}
    Y(L)
    &:=
    - (I_1^{r,\theta} + I_1^z) \\
    &=
    \nu \left(
    \| \partial_r\partial_{\theta} v \sqrt{\varphi_L} \|_{L^2(\Omega)}^2
    + \| \partial_z \partial_{\theta} v \sqrt{\varphi_L} \|_{L^2(\Omega)}^2 \right.\\
    &\left. \qquad
    + \left\| \frac{1}{r} (\partial_{\theta}^2 v^{\theta} + \partial_{\theta} v^r) \sqrt{\varphi_L}\right\|_{L^2(\Omega)}^2
    + \left\| \frac{1}{r} (\partial_{\theta}^2 v^r - \partial_{\theta} v^{\theta} ) \sqrt{\varphi_L} \right\|_{L^2(\Omega)}^2
    + \left\| \frac{1}{r} \partial_{\theta}^2 v^z \sqrt{\varphi_L} \right\|_{L^2(\Omega)}^2
    \right).
\end{align}
We remark that the definition of $\varphi_L(z)$ (see \eqref{varphiL}) implies
\begin{align}
    Y'(L)
    &:=
    \nu \left(
    \| \partial_r\partial_{\theta} v \|_{L^2(\Sigma_L)}^2
    + \| \partial_z \partial_{\theta} v \|_{L^2(\Sigma_L)}^2 \right.\\
    &\left. \qquad
    + \left\| \frac{1}{r} (\partial_{\theta}^2 v^{\theta} + \partial_{\theta} v^r) \right\|_{L^2(\Sigma_L)}^2
    + \left\| \frac{1}{r} (\partial_{\theta}^2 v^r - \partial_{\theta} v^{\theta} )  \right\|_{L^2(\Sigma_L)}^2
    + \left\| \frac{1}{r} \partial_{\theta}^2 v^z \right\|_{L^2(\Sigma_L)}^2
    \right).
\end{align}
Using the above $Y(L)$, we have
\begin{align}\label{eq:4:I:Y}
    I = -Y(L) + I_2^{r,\theta} + I_{2}^{z}.
\end{align}
Since $|\Sigma_L| = 2\pi(R_2^2 - R_1^2)$ is independent of $L$, 
it follows from the the Schwarz inequality and Lemma \ref{lem:bdd}  that 
the remainder terms $I_2^{z}$ is estimated as
\begin{align}
\label{eq:est:3:I}
    |I_2^{z}|
    &\le
    \nu\| \partial_z \partial_{\theta} v^{z} \|_{L^2(\Sigma_L)}
    \| \partial_{\theta} v^{z} \|_{L^2(\Sigma_L)} \\
    & \le
    \nu\|\nabla v\|_{L^\infty(\Omega)}|\Sigma_L|^{\frac12}
    \| \partial_z \partial_{\theta} v \|_{L^2(\Sigma_L)} \\
    & \le C\|\partial_z \partial_{\theta} v \|_{L^2(\Sigma_L)} \\
    &\le C \sqrt{Y'(L)},
\end{align}
with some constant $C>0$ independent of $L$. 
Similarly, it is easy to see that that $I_{2}^{r,\theta}$ has the same bound:
\begin{align}
\label{eq:est:3:I:rth}
	|I_2^{r,\theta}| \le C \sqrt{Y'(L)}.
\end{align}

%%%%%%%%%%%%% II
We next estimate the nonlinear term $I\!I$.
For
$\lambda = r,\theta,z$,
the integration by parts and the divergence free condition imply
\begin{align}
    I\!I^{\lambda}
    &=
    - \int_{\Omega}
    \left( \partial_{\theta} v \cdot \nabla \right) \partial_{\theta} v^{\lambda} v^{\lambda} \varphi_L(z)
    \,dx 
    - \int_{\Omega}
    \partial_{\theta} v^z v^{\lambda} \partial_{\theta} v^{\lambda} \partial_z \varphi_L(z) \,dx \\
    &=: I\!I_1^{\lambda} + I\!I_2^{\lambda}.
\end{align}
By the H\"{o}lder inequality, the term $I\!I_1^{\lambda}$ is estimated as
\begin{align}\label{eq:est:4:II_1}
    |I\!I_1^{\lambda}|
    &\le
    \| \partial_{\theta} v^r \sqrt{\varphi_L} \|_{L^2(\Omega)}
    \| \partial_r \partial_{\theta} v^{\lambda} \sqrt{\varphi_L} \|_{L^2(\Omega)}
    \| v \|_{L^{\infty}(\Omega)} \\
    &\quad
    +  \| \partial_{\theta} v^{\theta} \sqrt{\varphi_L} \|_{L^2(\Omega)}
    \left\| \frac{\partial_{\theta}^2 v^{\lambda}}{r} \sqrt{\varphi_L} \right\|_{L^2(\Omega)}
    \| v \|_{L^{\infty}(\Omega)} \\
    &\quad
    + \| \partial_{\theta} v^z \sqrt{\varphi_L} \|_{L^2(\Omega)}
    \| \partial_z \partial_{\theta} v^{\lambda} \sqrt{\varphi_L} \|_{L^2(\Omega)}
    \| v \|_{L^{\infty}(\Omega)}.
\end{align}
Let us further estimate the right-hand side.
From Lemma \ref{lem:Poincare}, we obtain for $\kappa = r,\theta,z$,
\begin{align}
\label{eq:poincare:2}
    \| \partial_{\theta} v^{\kappa} \sqrt{\varphi_L} \|_{L^2(\Omega)}
    &\le
    \sqrt{C_P}
    \| \partial_r \partial_{\theta} v^{\kappa} \sqrt{\varphi_L} \|_{L^2(\Omega)}.
\end{align}
Moreover, for the term
$\displaystyle \left\| \frac{\partial_{\theta}^2 v^{\lambda}}{r} \sqrt{\varphi_L} \right\|_{L^2(\Omega)}$,
by Lemma \ref{lem:Poincare},
we have, for the case $\lambda = r$,
\begin{align}
\label{eq:poincare:3}
    \qquad
    \left\| \frac{\partial_{\theta}^2 v^{r}}{r} \sqrt{\varphi_L} \right\|_{L^2(\Omega)}
    &\le 
    \left\| \frac{1}{r} \left( \partial_{\theta}^2 v^r - \partial_{\theta} v^{\theta} \right) \sqrt{\varphi_L} \right\|_{L^2(\Omega)}
    + \left\| \frac{1}{r} \partial_{\theta} v^{\theta} \sqrt{\varphi_L} \right\|_{L^2(\Omega)} \\
    &=
    \left\| \frac{1}{r} \left( \partial_{\theta}^2 v^r - \partial_{\theta} v^{\theta} \right) \sqrt{\varphi_L} \right\|_{L^2(\Omega)}
    + \frac{\sqrt{C_P}}{R_1} \left\| \partial_r \partial_{\theta} v^{\theta} \sqrt{\varphi_L} \right\|_{L^2(\Omega)};
\end{align}
and for the case $\lambda = \theta$,
\begin{align}
\label{eq:poinacre:4}
    \qquad
    \left\| \frac{\partial_{\theta}^2 v^{\theta}}{r} \sqrt{\varphi_L} \right\|_{L^2(\Omega)}
    &\le
    \left\| \frac{1}{r} \left( \partial_{\theta}^2 v^{\theta} + \partial_{\theta} v^{r} \right) \sqrt{\varphi_L} \right\|_{L^2(\Omega)}
    + \frac{\sqrt{C_P}}{R_1}
    \| \partial_r \partial_{\theta} v^{r} \sqrt{\varphi_L} \|_{L^2(\Omega)}.
\end{align}
Therefore, combining \eqref{eq:est:4:II_1}, \eqref{eq:poincare:2}, \eqref{eq:poincare:3}, and \eqref{eq:poinacre:4},
and applying the Schwarz inequality,
we deduce
\begin{align}
    |I\!I_1^r|
    &\le
    \frac{ \| v \|_{L^{\infty}(\Omega)} \sqrt{C_P}}{2}
    \left\{
    2 \| \partial_r \partial_{\theta} v^r \sqrt{\varphi_L} \|_{L^2(\Omega)}^2
    + \left( 1 + \frac{2C_P}{R_1^2} \right) \| \partial_r \partial_{\theta} v^{\theta} \sqrt{\varphi_L} \|_{L^2(\Omega)}^2
     \right. \\
    &\left. \quad
    + \| \partial_r \partial_{\theta} v^{z} \sqrt{\varphi_L} \|_{L^2(\Omega)}^2 
    + 2 \left\| \frac{1}{r} (\partial_{\theta}^2 v^r - \partial_{\theta} v^{\theta} )\sqrt{\varphi_L} \right\|_{L^2(\Omega)}^2
    +  \| \partial_z \partial_{\theta} v^r \sqrt{\varphi_L} \|_{L^2(\Omega)}^2
    \right\},\\
    |I\!I_1^{\theta}|
    &\le
    \frac{ \| v \|_{L^{\infty}(\Omega)} \sqrt{C_P}}{2}
    \left\{
    \left( 1 + \frac{2C_P}{R_1^2} \right) \| \partial_r \partial_{\theta} v^r \sqrt{\varphi_L} \|_{L^2(\Omega)}^2
    + 2 \| \partial_r \partial_{\theta} v^{\theta} \sqrt{\varphi_L} \|_{L^2(\Omega)}^2
     \right. \\
    &\left. \quad
    + \| \partial_r \partial_{\theta} v^{z} \sqrt{\varphi_L} \|_{L^2(\Omega)}^2 
    + 2 \left\| \frac{1}{r} (\partial_{\theta}^2 v^{\theta} + \partial_{r} v^{\theta} )\sqrt{\varphi_L} \right\|_{L^2(\Omega)}^2
    +  \| \partial_z \partial_{\theta} v^{\theta} \sqrt{\varphi_L} \|_{L^2(\Omega)}^2
    \right\},\\
    |I\!I_1^{z}|
    &\le
    \frac{ \| v \|_{L^{\infty}(\Omega)} \sqrt{C_P}}{2}
    \left\{
    \| \partial_r \partial_{\theta} v^r \sqrt{\varphi_L} \|_{L^2(\Omega)}^2
    + \| \partial_r \partial_{\theta} v^{\theta} \sqrt{\varphi_L} \|_{L^2(\Omega)}^2
     \right. \\
    &\left. \quad
    + 2\| \partial_r \partial_{\theta} v^{z} \sqrt{\varphi_L} \|_{L^2(\Omega)}^2 
    + \left\| \frac{1}{r} \partial_{\theta}^2 v^{z} \sqrt{\varphi_L} \right\|_{L^2(\Omega)}^2
    + \| \partial_z \partial_{\theta} v^{z} \sqrt{\varphi_L} \|_{L^2(\Omega)}^2
    \right\}.
\end{align}
Finally, adding the above estimates,
we conclude that the term
$I\!I_1 := I\!I_1^{r} + I\!I_1^{\theta} + I\!I_1^{z}$
satisfies the estimate
\begin{align}\label{eq:est:3:II}
    |I\!I_1|
    \le \| v \|_{L^{\infty}(\Omega)} \frac{C_{I\!I}}{\nu} Y(L)
\end{align}
with
\begin{align}
    C_{I\!I} 
    := \sqrt{C_P} \left( 2 + \frac{C_P}{R_1^2} \right).
\end{align}
By Lemma \ref{lem:bdd} and the Poinc\'are inequality in $\Sigma_L$we see that 
the remainder term $I\!I_2^{\lambda}$
can be estimated as
\begin{align}\label{eq:est:4:II_2}
    |I\!I_2^{\lambda}|
    &\le \| \partial_{\theta} v \|_{L^2(\Sigma_L)}^2 \| v \|_{L^{\infty}(\Sigma_L)} \\
    &\le \|\partial_\theta v \|_{L^\infty(\Omega)}|\Sigma_L|^{\frac12}\|v\|_{L^\infty(\Omega)}
    \|\pt_\theta v\|_{L^2(\Sigma_L)} \\
    & C \|\pt_r\pt_\theta v\|_{L^2(\Sigma_L)} \\
    & \le C \sqrt{Y'(L)}
\end{align}
for $\lambda = r,\theta,z$, where $C$ is the constant independent of $L$.  
\par
%%% III
Thirdly, we consider the nonlinear term $I\!I\!I$.
The integration by parts and the divergence free condition lead to
\begin{align}
    I\!I\!I^{\lambda}
    &=
    \int_{\Omega} (v\cdot \nabla) \left( \frac{1}{2} |\partial_{\theta} v^{\lambda}|^2 \right) \varphi_L(z) \,dx 
    =
    -\frac{1}{2} \int_{\Omega} v^z |\partial_{\theta} v^{\lambda}|^2 \partial_z \varphi_L(z) \,dx
\end{align}
for $\lambda = r,\theta,z$.
Moreover, similarly to \eqref {eq:est:4:II_2} we have that  
\begin{align}\label{eq:est:3:III}
    |I\!I\!I^{\lambda}|
    &\le
    C \| \partial_{\theta} v \|_{L^{2}(\Sigma_L)}^2 \| v \|_{L^{\infty}(\Sigma_L)}
    \le C \sqrt{Y'(L)}
\end{align}
for $\lambda = r,\theta,z$, where $C$ is a constant independent of $L$.  
%%%% IV
Next, we treat the remaining nonlinear term $I\!V$.
By Lemma \ref{lem:Poincare} and the Schwarz inequality,
we have
\begin{align}\label{eq:est:3:IV}
    |I\!V|
    &\le
    R_1^{-1}
    \left( \| \partial_{\theta} v^r \sqrt{\varphi_L} \|_{L^2(\Omega)}
        + \| \partial_{\theta} v^{\theta} \sqrt{\varphi_L} \|_{L^2(\Omega)} \right) \\
    &\quad \times
    \| \partial_{\theta} v^{\theta} \sqrt{\varphi_L} \|_{L^2(\Omega)}
    \| v \|_{L^{\infty}(\Omega)} \\
    &\le
    \| v \|_{L^{\infty}(\Omega)} \frac{C_P}{R_1}
    \left( \frac{1}{2} \| \partial_r \partial_{\theta} v^r \sqrt{\varphi_L} \|_{L^2(\Omega)}^2
        + \frac{3}{2} \| \partial_r \partial_{\theta} v^{\theta} \sqrt{\varphi_L} \|_{L^2(\Omega)}^2
    \right) \\
    &\le
    \| v \|_{L^{\infty}(\Omega)} \frac{C_{I\!V}}{\nu} Y(L).
\end{align}
where
\begin{align}
    C_{I\!V} := \frac{3C_P}{2R_1}.
\end{align}

Finally, we estimate the pressure terms $V$.
The integration by parts and the divergence free condition lead to
\begin{align}
    V =
    - \int_{\Omega} \partial_{\theta} p  \partial_{\theta} v^z \partial_z \varphi_L(z) \,dx.
\end{align}
Therefore, from Lemma \ref{lem:bdd} and the Poinc\'are inequality in $\Sigma_L$ 
we obtain
\begin{align}\label{eq:est:3:V}
    |V|
    &\le
    C \| \partial_{\theta} p \|_{L^2(\Sigma_L)} \| \partial_{\theta} v^z \|_{L^2(\Sigma_L)}
    \le
    C \| \partial_r \partial_{\theta} v^z \|_{L^2(\Sigma_L)}
    \le
    C \sqrt{Y'(L)}.
\end{align}

Now we put the estimates
\eqref{eq:4:I:Y}, \eqref{eq:est:3:I}, \eqref{eq:est:3:I:rth}, \eqref{eq:est:3:II}, \eqref{eq:est:4:II_2}, \eqref{eq:est:3:III}, \eqref{eq:est:3:IV}, and \eqref{eq:est:3:V}
together into \eqref{eq:4:I-V} and conclude
\begin{align}
    Y(L)
    &\le 
    \| v \|_{L^{\infty}(\Omega)} \frac{C_{I\!I} + C_{I\!V} }{\nu} Y(L) + C\sqrt{Y'(L)}.
\end{align}
We define
\begin{align}
    C_2(\nu,R_1,R_2) := \left[ \frac{C_{I\!I} + C_{I\!V} }{\nu} \right]^{-1}.
\end{align}
Then, by the assumption
$\| v \|_{L^{\infty}(\Omega)} < C_2(\nu, R_1, R_2)$,
we reach the differential inequality
\begin{align}
    Y(L) \le C \sqrt{Y'(L)}.
\end{align}
Therefore, the completely same argument as the previous section,
we have 
$Y(L) = 0$
for all $L>1$.
This implies that 
$\partial_r \partial_{\theta} v \equiv 0$,
that is,
$\partial_{\theta} v$
is independent of $r \in [R_1, R_2]$.
However, the boundary condition requires
$\partial_{\theta} v (R_j, \theta, z) = 0 \ (j=1,2)$
for any $(\theta, z) \in [0,2\pi] \times \mathbb{R}$.
Thus,
$\partial_{\theta} v$
must be identically zero.
Then, by \eqref{ns2:dtheta}, we see that
$\partial_r \partial_{\theta} p = \partial_{\theta}^2 p = \partial_z \partial_{\theta} p = 0$,
that is,
$\partial_{\theta} p \equiv c$ in $\Omega$ 
with some constant $c \in \mathbb{R}$.
However, since $p$ must be periodic in $\theta$,
we conclude $c = 0$,
that is, $\partial_{\theta} p \equiv 0$. 
This completes the proof of Theorem \ref{thm:liouville:2}.

\section*{Acknowledgement}
This work was supported by JSPS 	
Grant-in-Aid for Scientific Research (A)
Grant Number JP21H04433.  
\par
The authors declare no conflicts of interest.
Data sharing is not applicable to this article as no data were created or analyzed in this study.

\end{document}